\documentclass[11pt,a4paper]{article}

\usepackage{amsthm,amsmath,amsfonts,amssymb}
\usepackage[boxed,linesnumbered]{algorithm2e}
\usepackage{color}
\usepackage{float}
\usepackage{graphicx}
\usepackage{hyperref,url}
\usepackage{latexsym,epsfig}

\newcommand{\RR}{\mathbb{R}}

\newtheorem{definition}{Definition}
\newtheorem{remark}{Remark}

\DeclareMathOperator{\area}{area}

\newcommand{\spb}[1]{\smallskip}
\newcommand{\mpb}[1]{\medskip}
\newcommand{\bpb}[1]{\bigskip}

\renewcommand{\P}{{\cal P}}

\title{Decompositions of a rectangle into\\ non-congruent rectangles of equal area}
\author{C. Dalf\'o$^a$, M. A. Fiol$^b$, N. L\'opez$^c$, A. Mart\'{\i}nez-P\'erez$^d$\\ 
{\small $^a$Dept. de Matem\`atica, Universitat de Lleida}\\
{\small Igualada (Barcelona), Catalonia}\\
{\small {\tt cristina.dalfo@udl.cat}}\\
{\small $^{b}$Dept. de Matem\`atiques, Universitat Polit\`ecnica de Catalunya} \\
{\small Barcelona Graduate School of Mathematics} \\
{\small Barcelona, Catalonia} \\
{\small {\tt miguel.angel.fiol@upc.edu}} \\
{\small $^c$Dept. de Matem\`atica, Universitat de Lleida}\\
{\small Lleida, Spain}\\
{\small {\tt nacho.lopez@udl.cat}}\\
$^d${\small Dept. de An\'alisis Econ\'omico y Finanzas, Universidad de Castilla-La Mancha}\\
{\small Talavera de la Reina, Spain}\\
{\small{\tt {alvaro.martinezperez@uclm.es}}}
}

\newtheorem{proposition}{Proposition}[section]
\newtheorem{corollary}{Corollary}[section]
\newtheorem{lemma}{Lemma}[section]
\newtheorem{theorem}{Theorem}[section]
\newtheorem{conjecture}{Conjecture}[section]

\newtheorem{problem}{Problem}[section]

\newcommand\blfootnote[1]{%
	\begingroup
	\renewcommand\thefootnote{}\footnote{#1}%
	\addtocounter{footnote}{-1}%
	\endgroup
}

\begin{document}
\DeclareGraphicsExtensions{.jpg,.pdf,.mps,.png}

\date{}

\maketitle

\blfootnote{
\begin{minipage}[l]{0.28\textwidth} \includegraphics[trim=10cm 6cm 10cm 5cm,clip,scale=0.15]{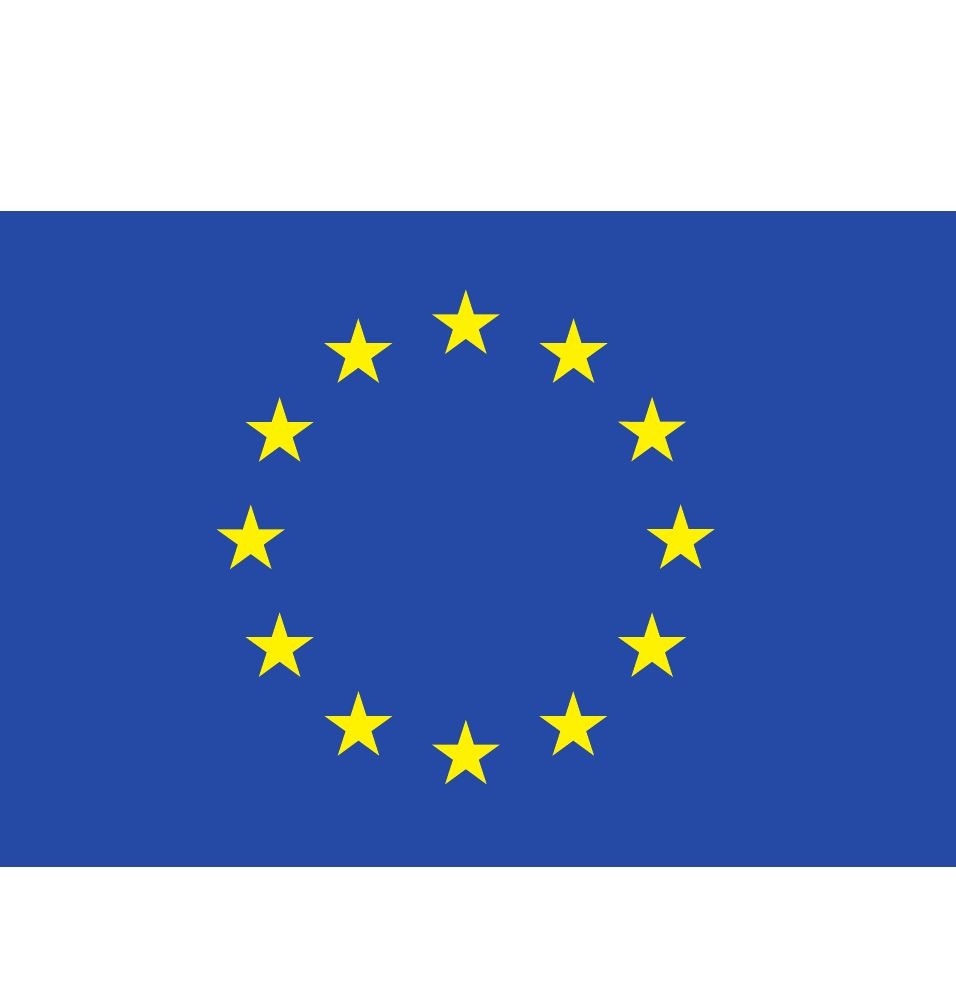} \end{minipage}  \hspace{-2cm} \begin{minipage}[l][1cm]{0.79\textwidth}
The research of the first author has also received funding from the European Union's Horizon 2020 research and innovation programme under the Marie Sk\l odowska-Curie grant agreement No 734922.
\end{minipage}
}

\begin{abstract}
In this paper, we deal with a simple geometric problem: Is it possible to partition a rectangle into $k$ non-congruent rectangles of equal area? This problem is motivated by the so-called `Mondrian art problem' that asks a similar question for dissections with rectangles of integer sides. Here, we generalize the Mondrian problem by allowing rectangles of real sides. In this case, we show that the minimum value of $k$ for a rectangle to have a `perfect Mondrian partition' (that is, with  non-congruent  equal-area rectangles) is seven. Moreover, we prove that such a partition is unique (up to symmetries) and that there exist exactly two proper perfect Mondrian partitions for $k=8$. Finally, we also prove that any square has a perfect Mondrian decomposition for $k\ge 7$.
\end{abstract}



\section{Introduction}
Geometric problems involving polygon decompositions arise in many situations. As an example, we can mention the problem of partitioning the unit square into $k$ rectangles of given areas $A_1,\dots,A_k$ (such that $\sum_{i=1}^k A_i = 1$) involving a specific optimization function (for instance, minimizing the largest perimeter of the $k$ rectangles). These geometric problems are related to the design of parallel matrix-matrix multiplication algorithms (see Beaumont, Boudet, Rastello, and Robert \cite{BBRR20}). Usually, the complexity of this kind of problems is at least NP-hard,  and many of them are NP-complete. Therefore, a strong research line for such problems is the design of `good' algorithms and heuristics to approximate their solutions.

Nevertheless, some decomposition problems have a different behavior. Given a positive integer $k$, consider the problem of decomposing a unit square into $k$ rectangles of area $1/k$ such that the maximum of their perimeters is minimized. Kong, Mount, and Werman \cite{KMW87} provided an elegant and constructive exact solution for this problem, where the rectangles of the optimal decomposition have all rational side lengths. Besides, the Mondrian art problem (see Bassen \cite{Bassen16} and O'Kuhn \cite{okuhn2018mondrian}) consists of partitioning a square of side $n \in \mathbb{N}$ into non-congruent rectangles of natural side lengths, such that the difference between the largest and smallest area of all rectangles is minimum. Here two rectangles are congruent if their corresponding side lengths,  not necessarily in the same order, are equal. An optimal solution to this problem would be to find a partition where every rectangle would have equal area. In this case, the difference between the largest and smallest area (named `defect') would be zero and, hence, we would have minimum defect. So far, this optimal partition has not been found for any square of side length $n \in \mathbb{N}$, and it seems that  such partition does not exist (the work contained in O'Kuhn \cite{okuhn2018mondrian} is partial progress towards answering this question).  This specific problem seems to be very hard since, in addition to the geometric nature of the problem, the dimensions of the geometric elements must be natural numbers. 

In this paper, we generalize the Mondrian art problem to consider partitions with rectangles of real sides. The results for $k\le 7$ are studied in Section
\ref{s:first-results}. In Section \ref{s:general-proper}, we propose a general procedure to deal with greater values of $k$. Then, it is conjectured that such a procedure always yields a so-called `proper perfect Mondrian partition' of a square.
Some enumeration results are then presented in Section \ref{s:enumeration}, where we completely describe the possible cases of proper perfect Mondrian partitions for $k=7$ and $k=8$. In Section \ref{s:perfect-rectangle} we prove that any square has a perfect Mondrian decomposition for $k\ge 7$. Finally, in the last section, we apply these results to the optimization problems explained above.

\section{The Mondrian problem with $7$ rectangles with real sides}
\label{s:first-results}

Given two rectangles, $R_1$ and $R_2$ (whose sides are real numbers), they are congruent, $R_1\cong R_2$, if the lengths of the sides of both rectangles (not necessarily in the same order) are equal.

Given any rectangle $R$, let us call a \emph{$k$-partition} of $R$ any tiling of $R$ with $k$ rectangles $R_1,\dots,R_k$ with $k\ge 2$. A partition of a rectangle $R$, $\P=\{R_1,\dots ,R_k\}$, is called a \emph{Mondrian partition} of $R$ if $R_i\not \cong R_j$ for every $i\neq j$. A Mondrian partition $\P=\{R_1,\dots ,R_k\}$ of a rectangle $R$ is \emph{perfect} if $\area(R_i)=\frac{1}{k}\,\area(R)$ for every $i$. (Notice that we are not asking that the sides of the rectangles $R_i$ are natural numbers.)


\begin{definition}
A partition $\P=\{R_1,\dots ,R_k\}$ of a rectangle $R$ is \emph{admissible} if no pair of rectangles have a common side.
\end{definition}

If $\P$ is a perfect Mondrian partition of a square $S$, then there is no pair of rectangles with a common side. Otherwise, since they have the same area, they would be congruent. Thus, we have the following observation.

\begin{remark}
\label{r:nec}
Given a Mondrian partition $\P$ of a rectangle, being admissible is a necessary condition for being perfect.
\end{remark}

\begin{definition}
A partition $\P=\{R_1,\dots ,R_k\}$ of a rectangle $R$ is \emph{proper} if it does not contain a subset  $\P'\subset  \P$ with $1< |\P'|<k$ forming a rectangle $R'$ contained in $R$.
\end{definition}
For instance, a proper partition  $\P(R)=\{R_1,\dots ,R_k\}$, with $k>2$, must be admissible, and  cannot have a rectangle $R_i$ with sides contained in two opposite sides of $R$.

\begin{remark}
\label{r:nec2}
Let $\P=\{R_1,\dots ,R_k\}$ be a non-proper admissible Mondrian partition of a rectangle $R$. Then, $\P$ contains an admissible partition
$\P'=\{R'_1,\dots ,R'_{k'}\}$ of some rectangle $R'$ with $k'\le k$.
\end{remark}

As a consequence of the above remarks, we then have the following results.

\begin{lemma}
\label{p:adm}
There is no admissible partition $\P=\{R_1,\dots ,R_k\}$ of a rectangle $R$ with $k\leq 4$.
\end{lemma}

\begin{proof}
If $k\leq 3$, it is trivial. If $k=4$, suppose $R_1$ and $R_2$ are two adjacent rectangles of the partition such that their intersection is not a common side. Then, either $R_3\cup R_4$ is a rectangle (if two sides of $R$ are contained in $R_1\cup R_2$) and $R_3$ and $R_4$ have a common side, or either $R_3$ or $R_4$ has a common side with $R_1$ or $R_2$.
\end{proof}

\begin{lemma}
\label{l:2intersec} If $\P=\{R_1,\dots, R_k\}$ is a partition of a rectangle $R$ with $k\leq 7$ such that no rectangle $R_i$ intersects two opposite sides of $R$, then there is a side of $R$ that intersects exactly two rectangles in $\P$.
\end{lemma}

\begin{proof}
Since no rectangle $R_i$ intersects two opposite sides of $R$, then every side of $R$ intersects at least two rectangles in $\P$. Suppose every side in $R$ intersects at least three rectangles in $\P$. Since no rectangle $R_i$ intersects two opposite sides of $R$, a rectangle $R_i$ intersects two different sides if and only if $R_i$ intersects one of the four corners of $R$. Thus, $\P$ has at least 8 rectangles.
\end{proof}

\begin{proposition}
\label{p:admis5}
There is no perfect Mondrian $5$-partition $\P=\{R_1,\dots ,R_5\}$ of a rectangle $R$. Moreover, there is a unique (up to symmetry) perfect admissible $5$- partition of $R$.
\end{proposition}

\begin{proof}
Suppose $\P=\{R_1,\dots ,R_5\}$ is a Mondrian partition of a rectangle $R$ and, for some $i$,
$R_i$ intersects opposite sides of $R$. Then, $R\setminus R_i$ contains a rectangle $R'$ with a partition $\P'\subset \P$ of $R'$ with less than five rectangles. Thus, by Lemma \ref{p:adm}, $\P'$ is not admissible, and, by Remark \ref{r:nec}, $\P$ is not perfect.\\
Let $\P=\{R_1,\dots ,R_5\}$ be an admissible partition of a rectangle  $R$.
We may assume that no rectangle $R_i$ intersects opposite sides of $R$. Then, by Lemma \ref{l:2intersec}, there is a side of $R$, which intersects exactly two rectangles in $\P$. Moreover, suppose there is a side of $R$ that intersects three rectangles in $\P$. Then, either two of them share a common side and the partition is not admissible, or it is immediate to check that it is impossible to form a rectangle with only two more rectangles. Thus, every side of $R$ intersects exactly two rectangles in $\P$, and there is a rectangle in $\P$, say $R_5$, that does not intersect any side of $R$. See Figure \ref{F: Case5a}.

\begin{figure}[ht]
\centering
\includegraphics[scale=0.5]{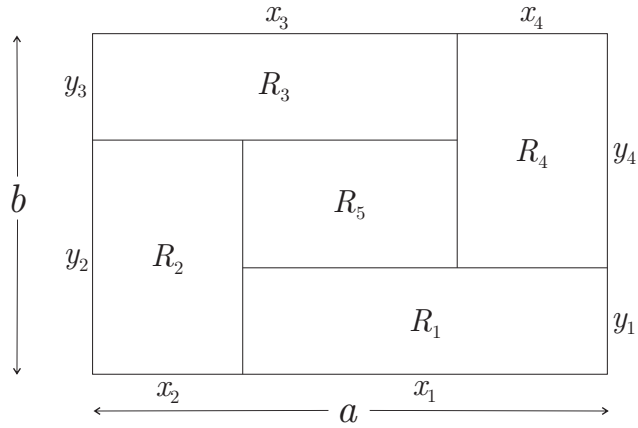}
\vskip-8cm
\caption{An admissible partition of a rectangle in five rectangles has one inner rectangle.}
\label{F: Case5a}
\end{figure}

Suppose $R$ has length sides $a,b\in \RR$, and $\P=\{R_1,\dots ,R_5\}$ is a perfect admissible partition of $R$, that is, $\area(R_i)=\frac{ab}{5}$ for every $i$. Consider $R_i$ with sides of length $x_i$ and $y_i$, as in Figure \ref{F: Case5a}. Then, $x_iy_i=\frac{a b}{5}$ for every $i$, and it is straightforward that $y_2$ depends on $x_1$, $x_3$ on $y_2$, and therefore on $x_1$, and so on. Thus, we have the following equations:
\[
\begin{aligned}
x_2 = & a-x_1, \quad \qquad \qquad \qquad y_2 = \frac{ab}{5x_2}=\frac{ab}{5a-5x_1};\\
y_3 = & b-y_2=\frac{4ab-5bx_1}{5a-5x_1}, \quad x_3= \frac{ab}{5y_3}=\frac{a(a-x_1)}{4a-5x_1};\\
x_4= & a-x_3 =\frac{3a^2-4ax_1}{4a-5x_1}, \quad y_4= \frac{ab}{5x_4}=\frac{4ab-5bx_1}{15a-20x_1}; \\
y_1= & b-y_4 =\frac{11ab-15bx_1}{15a-20x_1},\\
y_1= & \frac{ab}{5x_1}.
%
\end{aligned}
\]
Then, from the last two equations, we know that
\begin{equation}\label{eq:x1_5}
\frac{11ab-15bx_1}{15a-20x_1}=\frac{ab}{5x_1} \qquad \Leftrightarrow \qquad  5x_1^2-5ax_1+a^2=0,
\end{equation}
and there are only two solutions for this equation: $x_1=\frac{5+\sqrt{5}}{10}a$ or $x_1=\frac{5-\sqrt{5}}{10}a$.\\
Thus, if $x_1=\frac{5+\sqrt{5}}{10}a$, we have:
\begin{align*}
x_1 & =\frac{5+\sqrt{5}}{10}a=x_3, \quad y_1=\frac{5-\sqrt{5}}{10}b=y_3, \\
x_2 & =\frac{5-\sqrt{5}}{10}a=x_4, \quad y_2 =\frac{5+\sqrt{5}}{10}b=y_4.
\end{align*}
If $x_1=\frac{5-\sqrt{5}}{10}a$, we have:
\begin{align*}
x_1 & =\frac{5-\sqrt{5}}{10}a=x_3, \quad y_1=\frac{5+\sqrt{5}}{10}b=y_3, \\
x_2 & =\frac{5+\sqrt{5}}{10}a=x_4, \quad y_2=\frac{5-\sqrt{5}}{10}b=y_4.
\end{align*}
Notice that, in both cases, we have that $R_1\cong R_3$ and $R_2\cong R_4$, and  hence $\P$ is not Mondrian.
Besides, these two solutions are symmetric. See Figure \ref{F: Case5b}.

\begin{figure}[ht]
\centering
\includegraphics[scale=0.6]{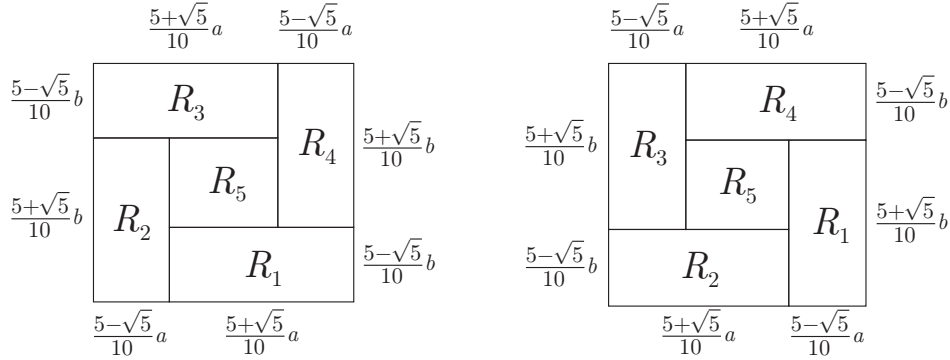}
\vskip-11.75cm
\caption{The unique (up to symmetry) perfect admissible partition of a rectangle into $k=5$ rectangles.}
\label{F: Case5b}
\end{figure}
\end{proof}

\begin{proposition}
There is no perfect Mondrian $6$-partition $\P=\{R_1,\dots ,R_6\}$ of a square $S$.
\end{proposition}

\begin{proof}
By Remark \ref{r:nec}, any perfect Mondrian partition must be admissible. Suppose $\P$ is an admissible Mondrian partition of $S$ and let us see that it is not perfect.\\
If there is a rectangle $R_i$ such that $R_i$ intersects two opposite sides of the square, then
$\P\setminus \{R_i\}$ contains an admissible partition of a rectangle contained in $S\setminus R_i$ with at most five rectangles. Thus, by  Proposition \ref{p:admis5}, this partition is not perfect.\\
Therefore, we may assume that, for every $i$, $R_i$ does not intersect two opposite sides of $S$. By Lemma \ref{l:2intersec}, there is a side of $S$, suppose it is the upper one, which intersects exactly two rectangles, say $R_5,R_6$. Since the partition $\P$ is admissible, the intersecting sides of $R_5$ and $R_6$ do not have the same length. Therefore, if $y_5$ and $y_6$ are the lengths of the intersecting sides of $R_5$ and $R_6$ respectively, 
we may assume, without loss of generality, that $y_5<y_6$. \\
Since the partition has six rectangles, it is straightforward that the lower side of $R_5$ cannot intersect more than two rectangles. Otherwise, since there is at most one rectangle left, at least two of them share a vertical side, leading to a contradiction. Also, it is immediate to check that it is not possible that the lower side of $R_5$ intersects only one rectangle since this would imply that this side is common, and the partition would not be admissible. Therefore, the lower side of $R_5$ intersects exactly two rectangles, say $R_3,R_4$, leaving two rectangles left.\\
\noindent Case 1. Suppose that the lower side of $R_6$ intersects two rectangles, $R_1$ and $R_2$.  Then, $R_1$ and $R_2$ share a vertical side leading to a contradiction. See Figure \ref{F: Case6-2}.

\begin{figure}[ht]
\centering
\includegraphics[scale=0.6]{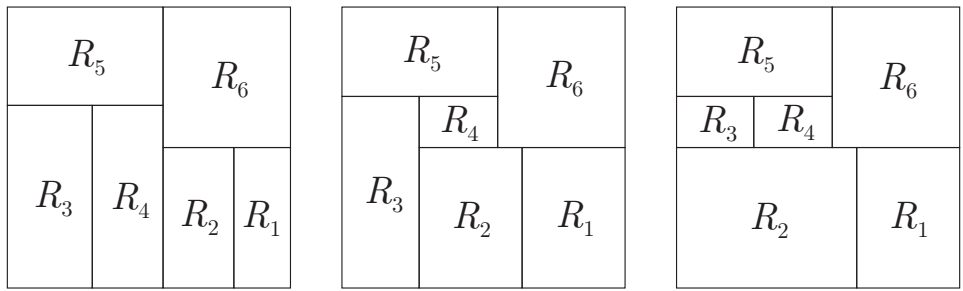}
\vskip-12.5cm
\caption{The lower side of $R_6$ intersects both rectangles left, $R_1$ and $R_2$.}
\label{F: Case6-2}
\end{figure}

\noindent Case 2. Suppose that the lower side of $R_6$ intersects only one rectangle, say $R_1$. Since $R_6$ and $R_1$ cannot have a common side, the horizontal side of $R_1$ is longer than the horizontal side of $R_6$. If the upper side of $R_1$ intersects the lower sides of $R_3$ and $R_4$, then necessarily $R_3$ and $R_4$ have a common side, see Figure \ref{F: Case6-1a}.

\begin{figure}[ht]
\centering
\includegraphics[scale=0.6]{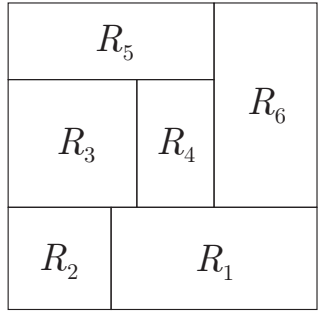}
\vskip-12cm
\caption{The upper side of $R_1$ intersects the lower sides of $R_3$ and $R_4$.}
\label{F: Case6-1a}
\end{figure}

Otherwise, it is immediate to check that $R_2$ has a common side either with $R_1$ or $R_3$, leading to a contradiction. See Figure \ref{F: Case6-1b}.

\begin{figure}[ht]
\centering
\includegraphics[scale=0.6]{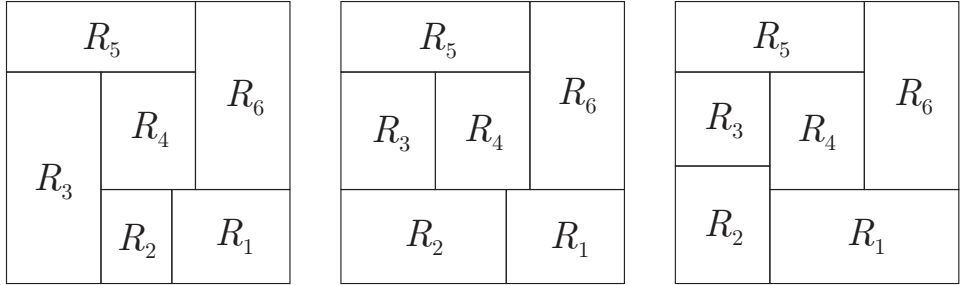}
\vskip-12.25cm
\caption{The upper side of $R_1$ does not intersect the lower side of $R_3$.}
\label{F: Case6-1b}
\end{figure}
\end{proof}

\begin{theorem}\label{T:div7}
Given a square $S$, there is a perfect Mondrian partition $\P(S)=\{R_1,R_2,\dots, R_7\}$.
\end{theorem}

\begin{proof}
First, let us consider the admissible partition of a square given in Figure \ref{F: Case7b}.

\begin{figure}[h]
\centering
\includegraphics[scale=0.6]{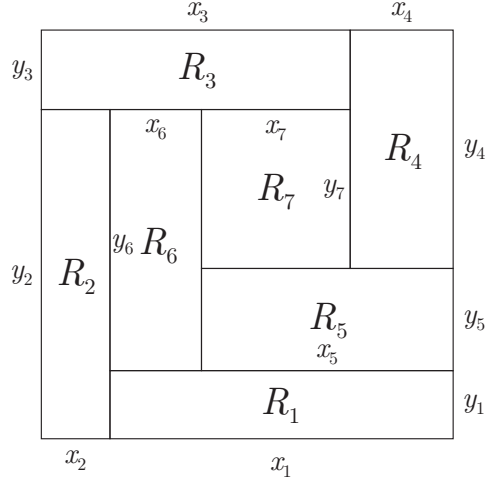}
\vskip-10cm
\caption{Admissible partition of a square with 7 rectangles.}
\label{F: Case7b}
\end{figure}

Let us assume that the side of the square is 1. Our partition  can be easily built so that $R_1,R_2,R_3,R_4$ and $R_5$ have area $\frac17$, that is, $x_iy_i=\frac{1}{7}$ for every $i\in \{1,2,3,4,5\}$. Then,  we have the following equations, where $x_2,x_3,x_4$  depend on $x_1$:
\[
\begin{aligned}
x_2 = & 1-x_1; \qquad \qquad \quad y_2=\frac{1}{7x_2}=\frac{1}{7-7x_1},\\
y_3= & 1-y_2=\frac{6-7x_1}{7-7x_1}, \quad x_3=\frac{1}{7y_3}=\frac{1-x_1}{6-7x_1},\\
x_4= & 1-x_3= \frac{5-6x_1}{6-7x_1},\quad
y_4=\frac{1}{7x_4}= \frac{6-7x_1}{7(5-6x_1)}.
\end{aligned}
\]
Moreover, if  $y_1+y_4<1$ $(\Rightarrow y_5>0)$, we have
\[
y_5=  1-y_1-y_4=\frac{5(7x_1^2-7x_1+1)}{7x_1(6x_1-5)},\quad
 x_5=\frac{1}{7y_5}=\frac{6x_1^2-5x_1}{35x_1^2-35x_1+5}.  \\
\]
Now, if the following conditions are also satisfied:
\begin{itemize}
	\item[$(i)$] $x_2+x_5<1$ $(\Rightarrow x_6>0)$,
	\item[$(ii)$] $y_1+y_3<1$ $(\Rightarrow y_6>0)$,
	\item[$(iii)$] $x_3>x_2+x_6$ $(\Rightarrow x_7>0)$,
    \item[$(iv)$] $y_4>y_3$ $(\Rightarrow y_7>0)$,
\end{itemize}
we can complete the partition of Figure \ref{F: Case7b}, and we have
\[
\begin{aligned}
x_6 = & 1-x_2-x_5=
\frac{35x_1^3-41x_1^2 +10x_1}{35x_1^2-35x_1+5},\quad y_6=1-y_1-y_3=\frac{1-2x_1}{7x_1(x_1-1)}.\\
\end{aligned}
\]
To prove that $R_6$ and $R_7$ also have area $\frac17$, it suffices to require that $x_6y_6=\frac{1}{7}$, which gives
\begin{equation*}\label{eq:R5}
\begin{aligned}
 & \left(\frac{35x_1^3-41x_1^2 +10x_1}{35x_1^2-35x_1+5}\right)\left(\frac{1-2x_1}{7x_1(x_1-1)}\right) =\frac{1}{7}\quad \
 \Leftrightarrow \quad
 105x_1^3-187x_1^2+101x_1-15  =0\\
  & \Leftrightarrow \quad
(7x_1-5)(15x_1^2-16x_1+3)= 0.
\end{aligned}
\end{equation*}

Thus, this last equation has three solutions: $s_1=\frac{5}{7}$,  $s_2=\frac{8-\sqrt{19}}{15}\approx 0.24274$, and $s_3=\frac{8+\sqrt{19}}{15}\approx 0.82393$.

It is immediate to check that solutions $s_1$ and $s_2$  do not produce the partition in Figure~\ref{F: Case7b}. Indeed,  if $x_1=s_1$, then $R_1\cong R_4$, $R_2\cong R_3$, and $x_7$ and $y_7$ have negative values; whereas,  if  $x_1=s_2$,  then $x_6$, $y_6$, $x_7$, and $y_7$ are also negative.

On the other hand, it is trivial to check (via computation) that the solution $x_1=s_3$ satisfies all the conditions.
To finish the proof, it suffices to check that $R_i\not \cong R_j$ for every $i\neq j$.  An approximation of the solution with $x_1=\frac{8+\sqrt{19}}{15}$ gives the values for $x_i$ and $y_i$ in Table \ref{t:sol-aprox(k=7)}.

\begin{table}[ht]
\centering
\begin{tabular}{|l|l|}
\hline\hline
$x_1 = 0.8239265962$  & $y_1 = 0.1733857646$\\
\hline
$x_2 =  0.1760734038$  & $y_2 =0.8113499245$\\
\hline
$x_3 = 0.7572599296,$ & $y_3 =0.1886500755$  \\
\hline
$x_4 = 0.2427400704$ & $y_4 =0.5885189973$\\
\hline
$x_5 =  0.6000000000$ & $y_5 = 0.2380952381$ \\
\hline
$x_6 = 0.2239265962 $ & $y_6 = 0.6379641599$ \\
\hline
$x_7 =0.3572599296$ & $y_7 =  0.3998689219$\\
\hline\hline
\end{tabular}
\caption{Approximate values for the partition with $k=7$.}
\label{t:sol-aprox(k=7)}
\end{table}

Thus, $x_1=\frac{8+\sqrt{19}}{15}$ gives a perfect Mondrian partition of the square with side length 1.
\end{proof}

\begin{table}[ht]
\centering
\begin{tabular}{|l|l|}
\hline\hline
$x_1 =  \frac{8+\sqrt{19}}{15}$ & $y_1 =  \frac{8-\sqrt{19}}{21} $ \\[.1cm]
\hline
$x_2 = \frac{7-\sqrt{19}}{15} $ & $y_2 = \frac{7+\sqrt{19}}{14}$   \\[.1cm]
\hline
$x_3 =\frac{7+\sqrt{19}}{15}$ & $y_3 = \frac{7-\sqrt{19}}{14} $ \\[.1cm]
\hline
$x_4 =  \frac{8-\sqrt{19}}{15}$ & $y_4 = \frac{8+\sqrt{19}}{21}$ \\[.1cm]
\hline
$x_5 = \frac{3}{5}$ & $y_5 = \frac{5}{21}$  \\[.1cm]
\hline
$x_6 =  \frac{-1+\sqrt{19}}{15} $ & $y_6 = \frac{5+5\sqrt{19}}{42}$ \\[.1cm]
\hline
$x_7 =  \frac{1+\sqrt{19}}{15}$ & $y_7 =  \frac{-5+5\sqrt{19}}{42} $ \\[.1cm]
\hline\hline
\end{tabular}
\caption{Exact values for the partition with $k=7$.}
\label{t:sol-exact(k=7)}
\end{table}

Table \ref{t:sol-exact(k=7)} shows the exact values of the approximations in Table \ref{t:sol-aprox(k=7)}. Notice that all the dimensions are irrational numbers, except for the ones of $R_5$.


\section{A general procedure}
\label{s:general-proper}

Generalizing the above section, we will give a general procedure to get a (proper) perfect Mondrian decomposition of a unit square into $k(\ge 7)$ rectangles.
This is obtained by means of the `spiral' pattern shown in Figure \ref{F: Case12} for $k=12$. Note that the partition of this figure is neither Mondrian (for instance, $R_1\cong R_2\cong R_3$) nor perfect, but the equations to be solved are the same when  considering such a pattern.
Moreover, notice that, for any rectangle $R$, this pattern provides a  proper admissible  partition $\P=\{R_1,\dots ,R_k\}$ of $R$ if and only if $k= 5$ or $k\geq 7$.

\begin{figure}[t]
	\centering
	\includegraphics[scale=0.5]{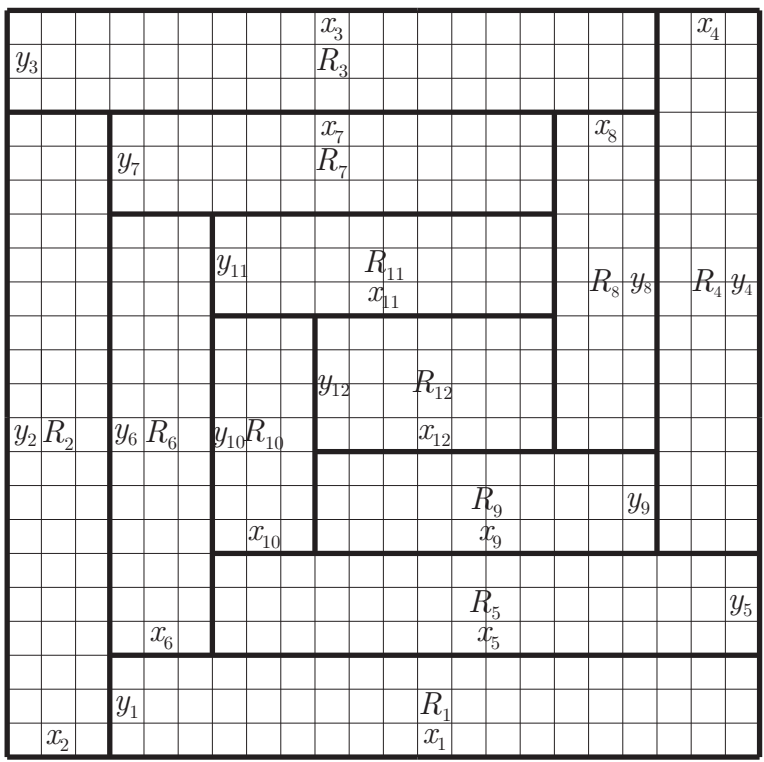}
	\vskip-3.5 cm
	\caption{The case with $k=12$.}
	\label{F: Case12}
\end{figure}

To establish the last equation to be solved, we proceed by imposing the right area $1/k$ of the rectangles following the spiral order: $R_1,R_2,R_3,\ldots, R_{k-1}$. Notice that, if this holds, the area of $R_k$ must also be $1/k$.
Then, the successive equalities, starting from step $1$, $x_1=x$ and  $y_1=\frac{1}{kx_1}$, are shown in Table \ref{t:eqs-general-proc}.

\begin{table}
\begin{center}
\begin{tabular}{|c|c|c|}
\hline\hline
Step & 1rt side & 2nd side \\
\hline\hline
$1$ & $x_1=x$ & $y_1=\frac{1}{kx_1}$\\[.1cm]
\hline
$2$ & $x_2=1-x_1$ & $y_2=\frac{1}{kx_2}$\\[.1cm]
\hline
$3$ & $y_3=1-y_2$ & $x_3=\frac{1}{ky_3}$\\[.1cm]
\hline
$4$ & $x_4=1-x_3$ & $y_4=\frac{1}{kx_4}$\\[.1cm]
\hline
$5$ & $y_5=1-y_1-y_4$ & $x_5=\frac{1}{ky_5}$\\[.1cm]
\hline
$6$ & $x_6=1-x_2-x_5$ & $y_6=\frac{1}{kx_6}$\\[.1cm]
\hline
$7$ & $y_7=1-y_1-y_3-y_6$ & $x_7=\frac{1}{ky_7}$\\[.1cm]
\hline
$8$ & $x_8=1-x_2-x_4-x_7$ & $y_8=\frac{1}{kx_8}$\\[.1cm]
\hline
$9$ & $y_9=1-y_1-y_3-y_5-y_8$ & $x_9=\frac{1}{ky_9}$\\[.1cm]
\hline
\vdots & \vdots & \vdots \\
\hline
$2n$ & $ x_{2n}=1-\displaystyle\sum_{i=1}^{n-2} x_{2i}-x_{2n-1}$ & $y_{2n}=\frac{1}{kx_{2n}}=1-\displaystyle\sum_{i=1}^{n-1}y_{2n-1}$ \\[.1cm]
\hline
$2n+1$ & $y_{2n+1}=1-\displaystyle\sum_{i=1}^{n-1} y_{2i-1}-y_{2n}$ & $x_{2n+1}=\frac{1}{ky_{2n+1}}=1-\displaystyle\sum_{i=1}^{n-1}x_{2i}$ \\[.1cm]
\hline\hline
\end{tabular}
\end{center}
\caption{The successive equations involved in the general case}
\label{t:eqs-general-proc}
\end{table}

Now, to find the correct value of $x(=x_1)$, we distinguish two cases, according to  the parity of $k$:
\begin{enumerate}
\item[$(i)$] $k$ {\bf odd}, $k=2n+1$: Find the largest solution of the equation (concerning the rectangle $R_{2n}$)
\begin{equation}
\label{sol-k-odd}
x_{2n}y_{2n}=\left(1-\sum_{i=1}^{n-2} x_{2i}-x_{2n-1}\right)
\left(1-\sum_{i=1}^{n-1}y_{2n-1}\right)=\frac{1}{k}.
\end{equation}
\item[$(ii)$] $k$ {\bf even}, $k=2n+2$: Find the largest solution of the equation (concerning the rectangle $R_{2n+1}$)
\begin{equation}
\label{sol-k-even}
x_{2n+1}y_{2n+1}=\left(1-\sum_{i=1}^{n-1}x_{2i}\right)
\left(1-\sum_{i=1}^{n-1} y_{2i-1}-y_{2n}\right)=\frac{1}{k}.
\end{equation}
\end{enumerate}

\begin{table}[h!]
\begin{center}
\begin{tabular}{|c|c|c|c| }
\hline\hline
$k$ & $p(x)$ & $z_1$ & $z_1\simeq$  \\
\hline\hline
$1$ & $x-1$  & 1 & 1 \\
\hline
$2$ & $2x-1$ &  1/2 & 0.5  \\
\hline
$3$ & $3x-2$ & 2/3 & 0.66666\ldots  \\
\hline
$4$ & $3x-2$   & 2/3 & 0.66666\ldots \\
\hline
$5$ & $5x^2-5x+1$   & $\frac{5+\sqrt{5}}{10}$   &  0.7236067977\ldots \\[.1cm]
\hline
$6$ & $24x^2-29x+8$   & $\frac{29+\sqrt{73}}{48}$  & 0.7821667446\ldots \\[.1cm]
\hline
$7$ & $105x^3-187x^2+101x-15$   & $\frac{3+\sqrt{19}}{15}$ & 0.8239265962\ldots \\[.1cm]
\hline
$8$ & $512x^4-1224x^3+1023x^2-342x+36$   &  \eqref{x1(k8)}  & 0.8520842333\ldots \\
\hline
$9$ & $15876x^5-50211x^4+60423x^3-33911x^2+8582x-735$  & --  & 0.8720043100\ldots \\
\hline
$10$ & $384000x^7-1605800x^6+2738920x^5-2435892x^4$  &  -- & \\
   & $+1196193x^3-315644x^2+40192x-1920$ &  & 0.8869492506\ldots \\
\hline\hline
\end{tabular}
\end{center}
\caption{The obtained polynomials for $k=1,\ldots,10$ and their largest zero $z_1$.}
\label{t:polys}
\end{table}

Although the products in \eqref{sol-k-odd} and  \eqref{sol-k-even} are rational functions on $x$, their solutions correspond to the zeros  of certain polynomials, as shown in Table \ref{t:polys}, together with the obtained values of $x_1$, as their largest root. For instance, when $k=8$, the obtained polynomial (of degree four) is depicted in Figure \ref{f:pol(k=8)}, and the two largest roots $x_1$ and $x_2$ (with $x_1>x_2$) are in the interval $(0.8,0.9)$. More precisely,
\begin{align}
z_1 &=\frac{1}{256}\left(153 -3\sqrt{41} + \sqrt{4386 + 426\sqrt{41}}\right)\simeq 0.8520842333,
\label{x1(k8)}\\
z_2 &=\frac{1}{256}\left(153 +3\sqrt{41} + \sqrt{4386 - 426\sqrt{41}}\right)\simeq 0.8317625891,
\label{x2(k8)}
\end{align}
where $z_1$ gives the right value of $x_1$. Indeed, from this, we get the approximate lengths of Table \ref{t:aprox-sol(k=8a)} for the eight rectangles shown in Figure \ref{f:case8a}. In fact, this is not the only solution, as we will show in
the next section.

\begin{figure}[h]
	\centering
	\vskip-.5cm
	\includegraphics[scale=0.5]{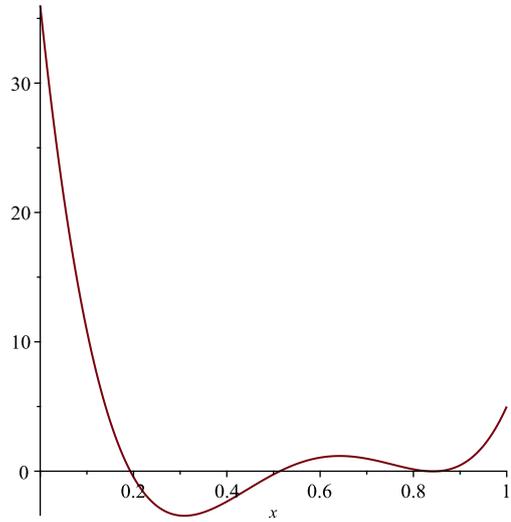}
	\vskip-6.1cm
	\caption{The polynomial for $k=8$.}
	\label{f:pol(k=8)}
\end{figure}


\begin{figure}[h!]
	\centering
	\vskip-2.5cm
	\includegraphics[scale=0.6]{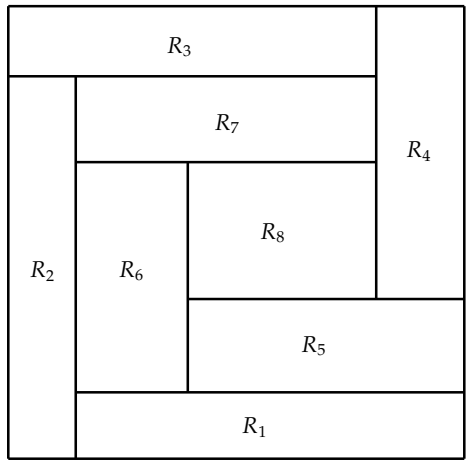}
	\vskip-7.75cm
	\caption{A perfect Mondrian partition of a square with $k= 8$ rectangles.}
	\label{f:case8a}
\end{figure}

\begin{table}[h!]
\centering
\begin{tabular}{|l|c|}
\hline\hline
$x_1 =0.8520842333$ & $y_1 =0.1466991115 $\\
\hline
$x_2 =0.1479157667$ & $y_2 =  0.8450755642$\\
\hline
$x_3 =0.8068449586 $ & $y_3 = 0.1549244358$\\
\hline
$x_4 =0.1931550414$ & $y_4 = 0.6471485242$\\
\hline
$x_5 =0.6063476421$ & $y_5 = 0.2061523643$\\
\hline
$x_6 = 0.2457365912 $ & $ y_6 = 0.5086747536 $\\
\hline
$x_7 =0.6589291919$ & $y_7 = 0.1897016991$\\
\hline
$x_8 =0.4131926007$ & $y_8 =  0.3025223893$\\
\hline\hline
\end{tabular}
\caption{Approximate lengths of the $k=8$ rectangles in a proper perfect Mondrian  partition of the  unit square.}
\label{t:aprox-sol(k=8a)}
\end{table}

This procedure leads us to formulate the following conjecture.

\begin{conjecture}
Given a square $S$, there is a proper perfect Mondrian partition $\P=\{R_1,\dots ,R_k\}$ of $S$ if and only if $k\geq 7$.
\end{conjecture}

\section{Some enumeration results}
\label{s:enumeration}

In this section, we prove that, up to symmetries, there is a unique proper perfect Mondrian of a rectangle for $k=7$, and exactly two  proper perfect Mondrian partitions for $k=8$.

Following a known approach (see Brooks,  Smith, Stone,
and Tutte \cite{Brooks1987}), we first represent every partition $\P=\{R_1,\ldots,R_n\}$ by a digraph $G(\P)$
with vertex set $V$ and arc set $A$ in the following way: Each vertex $P_i$ represents a horizontal (continuous) line containing (one or more) upper sides of adjacent rectangles in $\P$, and there is an arc from $P_i$ to $P_j$ if a rectangle $R_k\in \P$ has the upper side contained in (the line represented by) $P_i$, and the lower side contained in $P_j$.

Then, if $\P$ is a proper perfect partition, it is easy to check that the corresponding digraph $G(\P)$ has some forbidden `configurations'  (that is, subdigraphs), as shown in Figure \ref{f:forbidden}.

\begin{figure}[ht]
	\centering
	\includegraphics[scale=0.8]{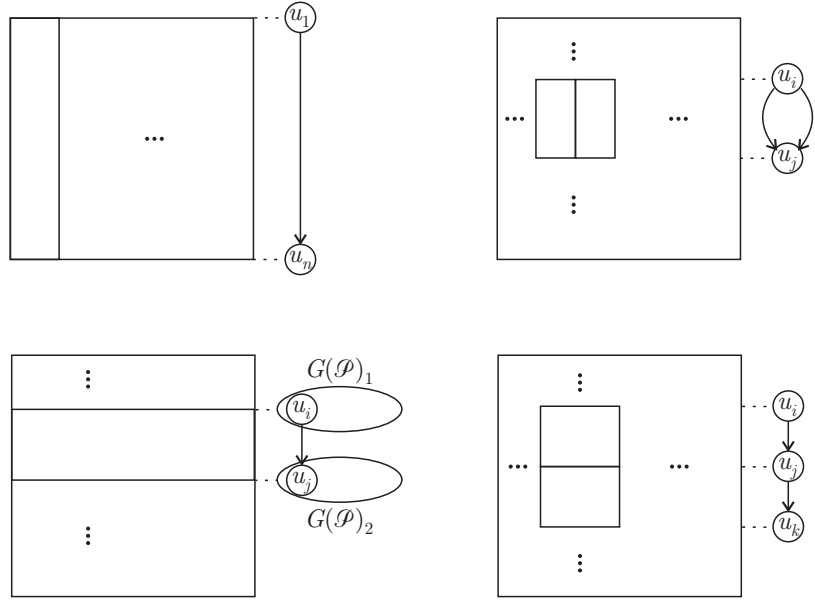}
	\vskip-14cm
	\caption{Forbidden configurations in a proper perfect partition. ($G(\P)_1$ and $G(\P)_2$ represent two disjoint subdigraphs of $G(\P)$.)}
	\label{f:forbidden}
\end{figure}

\begin{lemma}
\label{l:bound-n}
If $\P$ is a $k$-partition of a rectangle $R$, then there is a digraph $G(\P)=(V,A)$ with $k$ arcs, and number $n$ of vertices satisfying
\begin{equation}
\label{bounds-n}
\left\lceil \frac{k+7}{3}\right\rceil \le n\le \left\lfloor\frac{2(k+1)}{3} \right\rfloor.
\end{equation}
\end{lemma}

\begin{proof}
The first statement is a consequence of the definition of  $G(\P)$.
To prove \eqref{bounds-n}, we notice that,
to avoid the forbidden subdigraphs, the vertex $u_1$ has at least 2 outgoing arcs; the vertex $u_n$ has least 2 incoming arcs; and each vertex $u_i$ (with $i\neq 1,n$) has at least 3 (incoming or outcoming) arcs. Then,
\begin{equation}
\label{bounds-n2}
k=|A|\ge \frac{4+3(n-2)}{2}=\frac{3}{2}n-1,
\end{equation}
whence the upper bound in \eqref{bounds-n} follows because $n$ is an integer.

To find the lower bound, we consider the digraph $G'(\P)$ that corresponds to the same partition $\P$, but with vertices $v_1,\ldots, v_{n'}$ corresponding to the vertical sides (from left to right) of the partitioned rectangle $R$ (of course, this would be equivalent to rotate $90$ degrees clockwise $R$ to get $\P'$ and consider $G(\P')$). In fact, $G'(\P)$ can be obtained easily from $G(\P)$ in the following way: Consider the new vertices $v_1$ on the left of $G(\P)$, $v_{n'}$ on the right of $G(\P)$, and $v_2,\ldots, v_{n'-1}$ inside of every face of $G(\P)$. Moreover, put an arc from $v_i$ to $v_j$ (from left to right) if the corresponding directed edge crosses exactly one arc of $G(\P)$. Thus, $G(\P')$ is a kind of ``dual" of $G(\P)$ but, in the external face of it, we put two vertices ($v_1$ and $v_{n'}$) instead of one; see Figure \ref{f:cases7} (on the left), for an example. If $G(\P)$ has $n$ vertices, $c$ faces, and $k$ (directed) edges, then $G(\P')$ has $c'=n-1$ faces, $n'=c+1$ vertices, and $k'=k$ (directed) edges. Moreover, since $G(\P)$ and  $G(\P')$ are planar digraphs, Euler's formula and the upper bound in \eqref{bounds-n} yields
\begin{align*}
c+n=k+2 & \quad \Longrightarrow \quad n'-1+n=k+2 \\
        & \quad \Longrightarrow \quad
n'= k+3-n\ge k+3-\frac{2(k+1)}{3}=\frac{k+7}{3},
\end{align*}
and the lower bound follows.
\end{proof}

\subsection{The case $k=7$}
Now we can prove that the given proper perfect Mondrian $7$-partition is unique.
To this end, first notice that, by Lemma \ref{l:bound-n}, the number $n$ of vertices in the digraph $G(\P)$ is $5$.
Thus, we can assume that the digraph $G(\P)$ has vertices $u_1,\ldots,u_5$ and $k=7$ arcs (from top to bottom). Then, there are only two possible cases (up to symmetries), $G(\P_{7a})$ and $G(\P_{7b})$, are shown in Figure \ref{f:cases7}. (In fact, allowing symmetries, there are a total of eight cases. In the digraphs, such symmetries correspond  to take either the converse digraph, reversing all the arrows, or the ``dual digraph'', as explained before in the proof of Lemma \ref{l:bound-n}).  The first case on the left, with digraph $G(\P_{7a})$, corresponds to the spiral pattern studied in Section \ref{s:first-results}. In the second case, although the digraph $G(\P_{7b})$ is not isomorphic to $G(\P_{7a})$, it turns out that the equations to be solved give a forbidden (non-Mondrian) partition with three pairs of equal rectangles, $R_1\cong R_4$, $R_2\cong R_3$, and $R_6\cong R_7$, as
shown again in Figure \ref{f:cases7} (pattern with axial symmetry, as that of its digraph).
This allows to state the following result:
\begin{proposition}\label{prop:k=7}
There is only one (up to symmetries) proper perfect Mondrian $7$-partition of a square.
\end{proposition}
\begin{figure}[h!]
	\centering
	\includegraphics[scale=0.52]{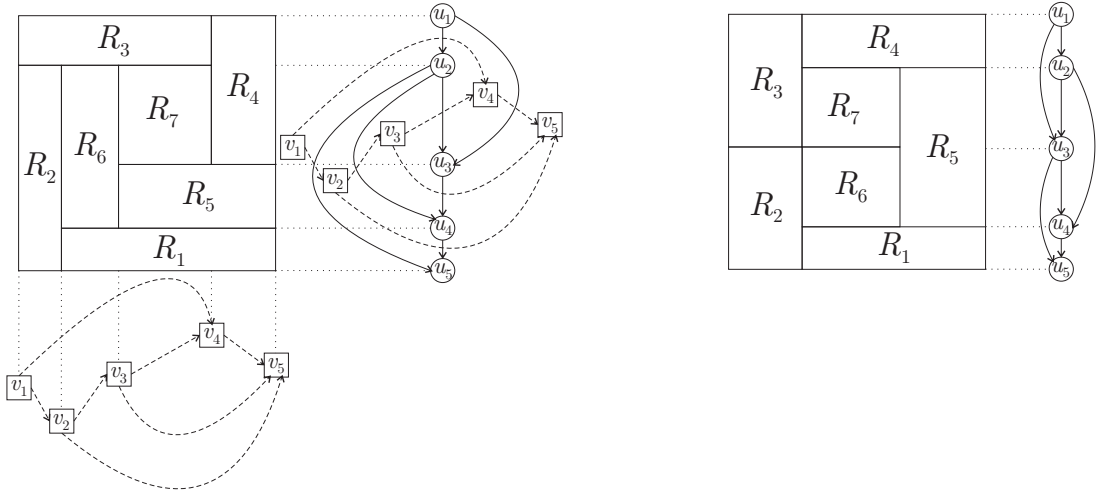}
	\vskip-4.25cm
	\caption{The two possible digraphs for the case $k=7$: $G(\P_{7a})$ (together with  its `dual digraph') on the  left, and $G(\P_{7b})$  on the right.}
	\label{f:cases7}
\end{figure}

\subsection{The case $k=8$}
A similar study, although a little more time-consuming, can be done for the case of proper  perfect Mondrian $8$-partitions.
In this case, by Lemma \ref{l:bound-n}, we see that it is enough to consider the digraphs $G(\P)$ with $n\in \{5,6\}$ vertices $u_1,u_2,\ldots$ and $k=8$ arcs. Then, the possible digraphs lead to essentially five different patterns: The first three partitions are spiral-like, as shown in Figure \ref{f:cases-spiral-8}, whereas the last two partitions have some strong symmetries, as shown in Figure \ref{f:cases-symmetric-8}. Notice that, in all the cases, the digraphs have 6 vertices; the (equivalent) ones with 5 vertices appear as dual of the former and, then, they do not need to be considered.

Notice that the first spiral pattern of Figure \ref{f:cases-spiral-8} corresponds to the solution found in Section \ref{s:general-proper} (see Figure \ref{f:case8a} and Table \ref{t:aprox-sol(k=8a)}).

\begin{figure}[h!]
	\centering
	\includegraphics[scale=0.45]{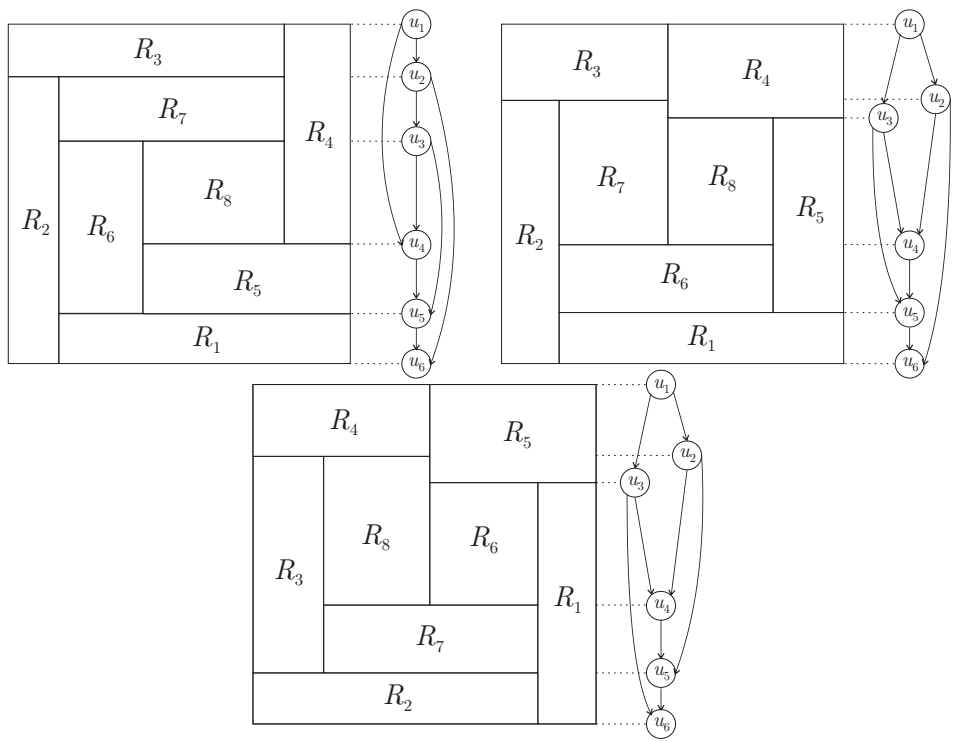}
	\vskip.25cm
	\caption{Possible spiral patterns for $k=8$, with  their corresponding digraphs:  $G(\P_{8a})$,  $G(\P_{8b})$, and $G(\P_{8b'})$.}
\label{f:cases-spiral-8}
\end{figure}

Moreover, the two spiral patterns  of Figure \ref{f:cases-spiral-8}, in the top-right and bottom, $G(\P_{8b})$, and $G(\P_{8b'})$, are equivalent in the sense that they give rise to the same equations (when we try to find a proper perfect Mondrian partition). In fact, it is easy to check that the equations for $x_i$ and $y_i$, for $i=1,\ldots,7$, are exactly the same as the ones for the `standard' spiral pattern of Figure \ref{f:case8a}. For instance, in both cases, we have $x_7=1-x_2-x_4$ and $y_7=1-y_1-y_3-y_6$ (following the general  procedure of Section \ref{s:general-proper}, which is summarized in Table \ref{t:eqs-general-proc}). Consequently, the polynomial (whose zeros are the candidates for giving a proper perfect Mondrian $8$-decomposition) is the same as the one shown in Table \ref{t:polys} for $k=8$ (see also Figure \ref{f:pol(k=8)}, with largest roots $x_1$ and $x_2$ in \eqref{x1(k8)} and \eqref{x2(k8)}).
However, curiously enough, the final pattern is not the same because now the sides of $R_8$ are:
\begin{align*}
x_8 & =1-x_3-x_5(=1-x_2-x_5-x_7),\\
y_8 &=1-y_1-y_4-y_6,
\end{align*}
and the zero to be taken now as the value of $x_1$ is not $z_1$, but $z_2\simeq 0.8317625891$. Thus, with $x_1=z_2$, we obtain
the solution shown in Figure \ref{f:case8b}, with approximate values of the rectangle sides in Table \ref{t:aprox-sol(k=8b)}.

\begin{figure}[h!]
	\centering
	\vskip-1.5cm
	\includegraphics[scale=0.5]{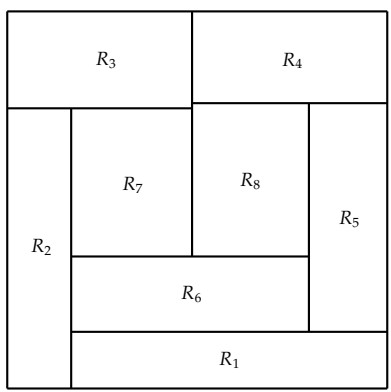}
	\vskip-6.5cm
	\caption{A second proper perfect Mondrian  $8$-partition of the unit square.}
	\label{f:case8b}
\end{figure}

\newpage

\begin{table}[h!]
\centering
\begin{tabular}{|l|c|}
\hline\hline
$x_1 =0.8317625891$ & $y_1 =0.1502832679$\\
\hline
$x_2 =0.1682374109$ & $y_2 =0.7429976444$\\
\hline
$x_3 =0.4863768650$ & $y_3 = 0.2570023556$\\
\hline
$x_4 =0.5136231350$ & $y_4 =0.2433690998$\\
\hline
$x_5 =0.2061523676$ & $y_5 =0.6063476323$\\
\hline
$x_6 =0.6256102215$ & $ y_6 =0.1998049196$\\
\hline
$x_7 =0.3181394541$ & $y_7 =0.3929094569$\\
\hline
$x_8 =0.3074707674$ & $y_8 =0.4065427127$\\
\hline\hline
\end{tabular}
\caption{Approximate lengths of the $k=8$ rectangles in the second proper perfect Mondrian partition of the  unit square.}
\label{t:aprox-sol(k=8b)}
\end{table}

\begin{figure}[h!]
	\centering
	\includegraphics[scale=0.5]{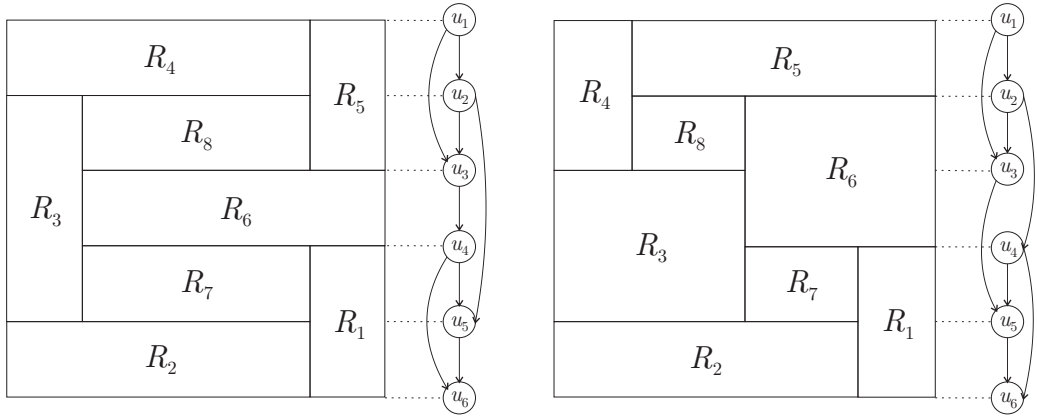}
	\vskip-5cm
	\caption{Possible  patterns for $k=8$ with  (specular and central)  symmetry.}
	\label{f:cases-symmetric-8}
\end{figure}

Finally,  the  two last digraphs leading to the partitions, with specular and central symmetries, of  Figure \ref{f:cases-symmetric-8}
do not lead to proper perfect Mondrian partitions because the obtained solutions do not break such symmetries and, hence, the partitions are not Mondrian (that is, as in the second case of $k=7$, there are pairs of equal rectangles).

Summarizing, we get the following result.
\begin{proposition}\label{prop:k=8}
There are only two (up to symmetries) proper perfect Mondrian $8$-partitions of a square.
\end{proposition}


\section{Perfect Mondrian partitions of a rectangle/square}
\label{s:perfect-rectangle}

Although we have no proof that the  procedure of Section \ref{s:general-proper} always yields a proper perfect Mondrian partition for every $k\ge 7$, we will prove now that, for $k\ge 8$, this is the case when we remove the `proper' condition.
First, we need the following lemma.

\begin{lemma}
\label{l:cond-R->R'}
Let $\P=\{R_1,\ldots R_k\}$ be a perfect Mondrian $k$-partition of a rectangle $R$ with sides $a$ and $b$. Let $x_i$ and $y_i$ be the dimensions of $R_i$ for $i=1,\ldots,k$.
Let $R'$ be a rectangle with dimensions $a'$ and $b'$. Then, there exists a perfect Mondrian $k$-partition $\P'=\{R'_1,\ldots R'_k\}$  of $R'$, with $x'_i$ and $y'_i$ being the dimensions of $R'_i$, for  $i=1,\ldots,k$, if
\begin{equation}
\label{cond-R->R'}
x_ix_j\neq \frac{a^2}{k}\frac{b'}{a'}\quad \mbox{for every $i\neq j$}.
\end{equation}
\end{lemma}
\begin{proof}
To prove the result, we simply re-scale the dissected rectangle $R$ to get the partition of $R'$. This requires to multiply all the horizontal sides of $R$ by $\frac{a'}{a}$, and all its vertical sides by $\frac{b'}{b}$. Now, it is clear that $\P'$ is also perfect, and we only need to prove that it is Mondrian, that is, $R'_i\ncong R'_j$ for $i\neq j$. Since $x_i\neq x_j$ and $y_i\neq y_j$, we obviously have $x'_i\neq x'_j$ and $y'_i\neq y'_j$, for every $i\neq j$. Moreover, let us check that, for every $i\neq j$, we also have
$x'_i\neq y'_j$. That is,
$$
\frac{a'}{a}x_i\neq\frac{b'}{b}y_j=\frac{b'}{b}\frac{ab}{kx_j},
%
$$
in concordance with the required condition \eqref{cond-R->R'}.
\end{proof}

In particular, if $R'=S$ is a square, $a'=b'$,
then the condition \eqref{cond-R->R'} becomes
\begin{equation}
\label{cond-R'->S}
x_ix_j\neq \frac{a^2}{k}\quad \mbox{for every $i\neq j$}.
\end{equation}

\begin{theorem}
\label{theorem5-1}
Let $\P_7=\{R_1,\ldots,R_7\}$ be the perfect $7$-partition of a unit square, in Theorem \ref{T:div7}, and with dimensions $x_i$ of $R_i$, for $i=1,\ldots,7$, in Table \ref{t:sol-exact(k=7)}. Then, the following statements are satisfied.
\begin{itemize}
\item[$(i)$]
Every rectangle $R'$ with dimensions $a'$ and $b'$ satisfying $\frac{b'}{a'}\neq 7x_ix_j$, for every $i\neq j$, admits a perfect Mondrian partition with $k=7$.
\item[$(ii)$]
Every square $S$ admits a perfect Mondrian partition with $k\ge 8$.
\end{itemize}
\end{theorem}
\begin{proof}
To prove $(i)$, we just apply  Lemma \ref{l:cond-R->R'} with $a=1$.
In the case $(ii)$, we start again from $\P_7$, and we get the successive perfect Mondrian partitions $\P'_{k}$, with $k>7$, of a unit square $S$, by doing the following steps:
\begin{itemize}
\item
$\P'_7=\P_7$.
\item
$\P'_8$ is obtained from $\P'_7$ by adding a rectangle on the bottom
with dimensions $x'_8=1$ and $y'_8=1/7$, and re-scaling the rectangle by multiplying all the vertical sides $y'_i$ by $7/8$.
\item
$\P'_9$ is obtained from $\P'_8$ by adding a rectangle on the left
with dimensions $y'_9=1$ and $x'_9=1/8$, and re-scaling the rectangle by multiplying all the horizontal sides $x'_i$ by $8/9$.
\item
$\P'_{10}$ is obtained from $\P'_9$ by adding a rectangle on the bottom
with dimensions $x'_{10}=1$ and $y'_{10}=1/9$, and re-scaling the rectangle by multiplying all the vertical sides $y'_i$ by $9/10$.
\item[] $\vdots$
\item
$\P'_k$ ($k$ odd) is obtained from $\P'_{k-1}$ by adding a rectangle on the left with dimensions $y'_{k}=1$ and $x'_{k}=1/(k-1)$, and re-scaling the rectangle by multiplying all the horizontal sides $x'_i$ by $(k-1)/k$.
Besides,
$\P'_{k+1}$ ($k+1$ even) is obtained from $\P'_{k}$ by adding a rectangle on the bottom
with dimensions $x'_{k+1}=1$ and $y'_{k+1}=1/k$, and re-scaling the rectangle by multiplying all the vertical sides $y'_i$ by $k/(k+1)$.
\end{itemize}

See an example of this method in Figure \ref{f:cases-k=13}.

\begin{figure}[h!]
	\centering
	\includegraphics[scale=0.55]{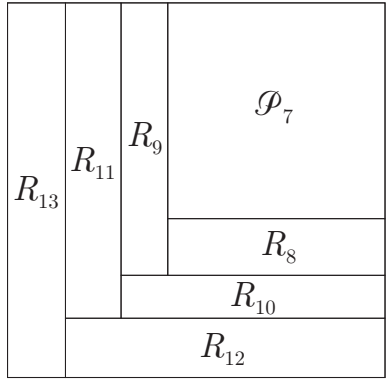}
	\vskip-9.5cm
	\caption{An example of the method used in the proof of Theorem \ref{theorem5-1} for $k=13$. For simplicity of the drawing, the areas of the rectangles are not equal.}
	\label{f:cases-k=13}
\end{figure}

Thus, depending on the parity of $k$, the partition $\P'_k$ has rectangles with the  dimensions shown in Table \ref{dim-k-even} for $k$ even, and Table \ref{dim-k-odd} for $k$ odd.
\begin{table}[h!]
\begin{center}
\begin{tabular}{|c|c|c|}
\hline\hline
Rectangle &  Horizontal side $x'_i$ & Vertical side $y'_i$ \\
\hline\hline
$R'_i (1\le i\le 7)$ &
$x'_i=\frac{8\cdot 10\cdots (k-2)}{9\cdot 11\cdots (k-1)}x_i$ &
$y'_i=\frac{7\cdot 9\cdots (k-1)}{8\cdot 10\cdots k}y_i$\\[.1cm]
\hline
$R'_8$ &
$x'_8=\frac{8\cdot 10\cdots (k-2)}{9\cdot 11\cdots (k-1)}$ &
$y'_8=\frac{9\cdot 11\cdots (k-1)}{8\cdot 10\cdots k}$ \\[.1cm]
\hline
$R'_9$ &
$x'_9=\frac{10\cdot 12\cdots (k-2)}{9\cdot 11\cdots (k-1)}$ &
$y'_9=\frac{9\cdot 11\cdots (k-1)}{10\cdot 12\cdots k}$ \\[.1cm]
\hline
\vdots & \vdots & \vdots \\
\hline
$R'_{i}$ ($i$ even) &
$x'_{i}=x'_{i-1}(i-1)$  & $y'_{i}=\frac{y'_{i-1}}{i-1}$\\[.1cm]
\hline
$R'_{i}$ ($i$ odd) &
$x'_{i}=\frac{x'_{i-1}}{i-1}$  & $y'_{i}=y'_{i-1}(i-1)$\\[.1cm]
\hline
\vdots & \vdots & \vdots \\
\hline
$R'_{k-1}$ &
$x'_{k-1}=\frac{1}{k-1}$ & $y'_{k-1}=\frac{k-1}{k}$ \\[.1cm]
\hline
$R'_{k}$ & $x'_k=1$  & $y'_k=\frac{1}{k}$ \\[.1cm]
\hline\hline
\end{tabular}
\end{center}
\caption{Dimension of the rectangles in $\P'_k(S)$ when $k$ is even.}
\label{dim-k-even}
\end{table}

\begin{table}[h!]
\begin{center}
\begin{tabular}{|c|c|c|}
\hline\hline
Rectangle &  Horizontal side $x'_i$ & Vertical side $y'_i$ \\
\hline\hline
$R'_i (1\le i\le 7)$ &
$x'_i=\frac{8\cdot 10\cdots (k-1)}{9\cdot 11\cdots k}x_i$ &
$y'_i=\frac{7\cdot 9\cdots (k-2)}{8\cdot 10\cdots (k-1)}y_i$\\[.1cm]
\hline
$R'_8$ &
$x'_8=\frac{8\cdot 10\cdots (k-1)}{9\cdot 11\cdots k}$ &
$y'_8=\frac{9\cdot 11\cdots (k-2)}{8\cdot 10\cdots (k-1)}$ \\[.1cm]
\hline
$R'_9$ &
$x'_9=\frac{10\cdot 12\cdots (k-1)}{9\cdot 11\cdots k}$ &
$y'_9=\frac{9\cdot 11\cdots (k-2)}{10\cdot 12\cdots (k-1)}$ \\[.1cm]
\hline
\vdots & \vdots & \vdots \\
\hline
$R'_{i}$ ($i$ even) &
$x'_{i}=x'_{i-1}(i-1)$  & $y'_{i}=\frac{y'_{i-1}}{i-1}$\\[.1cm]
\hline
$R'_{i}$ ($i$ odd) &
$x'_{i}=\frac{x'_{i-1}}{i-1}$  & $y'_{i}=y'_{i-1}(i-1)$\\[.1cm]
\hline
\vdots & \vdots & \vdots \\
\hline
$R'_{k-1}$ &
$x'_{k-1}=\frac{k-1}{k}$ &
$y'_{k-1}=\frac{1}{k-1}$ \\[.1cm]
\hline
$R'_{k}$ & $x'_k=\frac{1}{k}$  & $y'_k=1$ \\[.1cm]
\hline\hline
\end{tabular}
\end{center}
\caption{Dimension of the rectangles in $\P'_k(S)$ when $k$ is odd.}
\label{dim-k-odd}
\end{table}
Notice that if, for example $k$ is even, the whole square $S$ has horizontal side $a=x'_k=1$ and vertical side $b=y'_k+y'_{k-1}=\frac{k-1}{k}+\frac{1}{k}=1$; and similarly when $k$ is odd. Moreover, in both cases, $\area(R'_i)=x'_iy'_i=1/k$, for every $i=1,\ldots,k$, as required.
To finish the proof we only need to check that $R'_i\ncong R'_j$, that is, $(a)$ $x'_i\neq x'_j$ for $i\neq j$; and $(b)$ $x'_i\neq y'_j$ for any $i,j=1,\ldots,k$.
Suppose that $k$ is even (the case of odd $k$ is proved similarly). The first condition $(a)$
clearly holds if $i,j\in[1,7]$ (since $\P'_7$ is perfect) or $i,j\in[8,k]$ (by construction). Moreover, if $i\in[1,7]$ and $j\in[8,k]$, $x_i$, which is not rational, except for $x'_5$, cannot be equal to $x'_j$, which is rational. If $i=5$, $x'_{k-1},x'_k\neq x'_5$. In the other cases, it is immediate checking tables \ref{dim-k-even} and \ref{dim-k-odd} that $x'_j$ cannot be equal to $x'_5$.
Concerning the second condition $(b)$, we have $x'_i\neq y'_j$ if $i,j\in[1,7]$ (since $\P'_7$ is perfect), or if $i,j\in [8,k]$ (checking directly for $k-1$ and $k$, and using the parity argument otherwise, that is, $x'_i=\frac{E}{O}$ and $y_j=\frac{O}{E}$). Finally, if $x'_i\in [1,7]$ and $y'_j\in [8,k]$ or vice-versa, the inequality $x'_i\neq y'_j$ follows either by inspection (if $i=5$ or $j=5$) or since  the non-rational and rational numbers are mutually exclusive.
\end{proof}

\section{Applications to optimization problems} 
\subsection{The Mondrian art problem}\label{ex:div} We recall that the Mondrian art problem consists in partitioning a square of side $n \in \mathbb{N}$ into non-congruent rectangles of natural side lengths, such that the defect (difference between the largest and smallest area of all rectangles) is minimum. From the results of Section \ref{s:first-results}, and since in the cases $k=7$ and $k=8$ the obtained values for the sides of the rectangles are not rational, we get the following result about the Mondrian art problem. That is, to have a perfect Mondrian dissection of a square with rectangles of integer sides.
\begin{corollary}
The is no perfect integer Mondrian partition of a square with a number  $k\le 8$ of  rectangles.
\end{corollary}
Nevertheless, using the approximate values of Table \ref{t:sol-aprox(k=7)}, we can build admissible partitions of a square in seven rectangles with all the sides being natural numbers. See, for example, for a square of side 100, we have the example of Figure \ref{F:Ejp1} (a) with a defect of 74, this is, 74\% of the side.

\begin{figure}[h]
\centering
\includegraphics[scale=0.6]{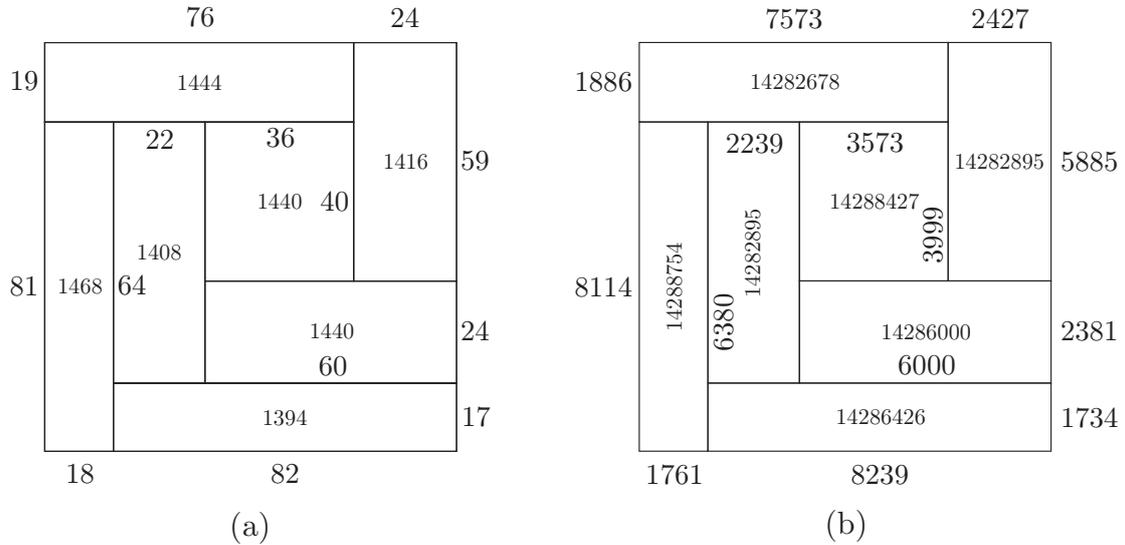}
\vskip-5.25cm
\caption{(a) A partition for the Mondrian art problem taking $n=100$ and defect $74$. (b) Another example for $n=1000$ and defect $6076$.}
\label{F:Ejp1}
\end{figure}

Also, considering a square of side 10000, we have the example of Figure \ref{F:Ejp1} (b) with a difference of 6076, this is less than 61\% of the side.



\subsection{Decompositions of the unit square in non-congruent rectangles where the maximum of their perimeters is minimized}
The problem of decomposing a unit square into $k$ rectangles of area $1/k$ such that the maximum of the perimeters of the rectangles is minimized was solved by Kong, Mount and Werman in \cite{KMW87}. The optimal decomposition provided by them contains many congruent rectangles for any $k$. So, it is natural to ask for those optimal decompositions using non-congruent rectangles.


\begin{problem}\label{prob:dec}
Given a positive integer $k$, decompose a unit square into $k$ non-congruent rectangles of area $1/k$ such that the maximum of their perimeters is minimized.
\end{problem}

Due to the results given in Section \ref{s:first-results}, there is no solution to problem \ref{prob:dec} for $k<7$. Moreover, for $k=7$, there is a unique decomposition (up to symmetries, see Figure \ref{F: Case7b}) and the maximum of the perimeters of the partition is given by $2(x_1+y_1)=\frac{192+4\sqrt{19}}{105} \approx 1.9946$ (see Table \ref{t:sol-exact(k=7)}). For $k=8$, there are just two decompositions (Proposition \ref{prop:k=8}) and the maximum of the perimeters of the rectangles is minimized in the second case. Here the dimensions of the largest rectangle are $x_1$ and $y_1$ given in Table \ref{t:aprox-sol(k=8b)}, obtaining the perimeter value $2(x_1+y_1)\approx 1.9641$.

\begin{figure}[htb]
	\centering
	\begin{tabular}{ccc}
	\includegraphics[scale=0.65]{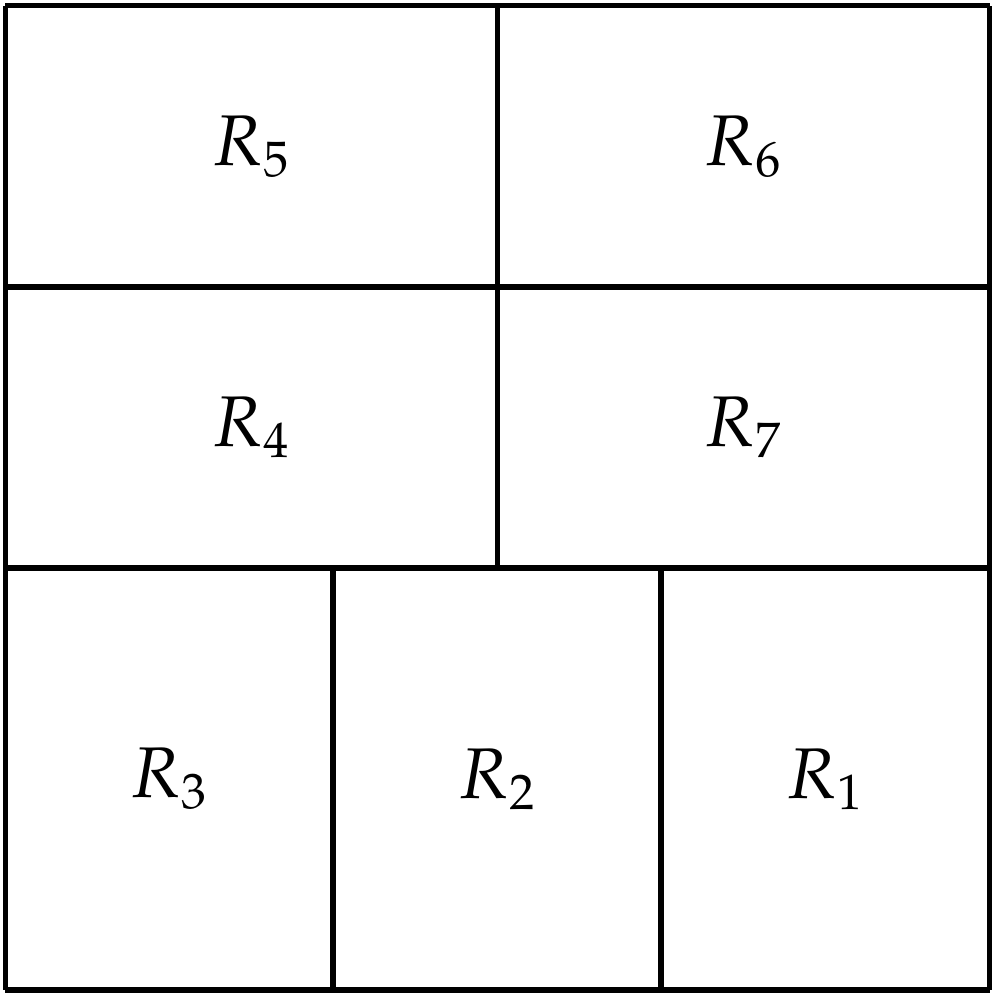} & & \includegraphics[scale=0.65]{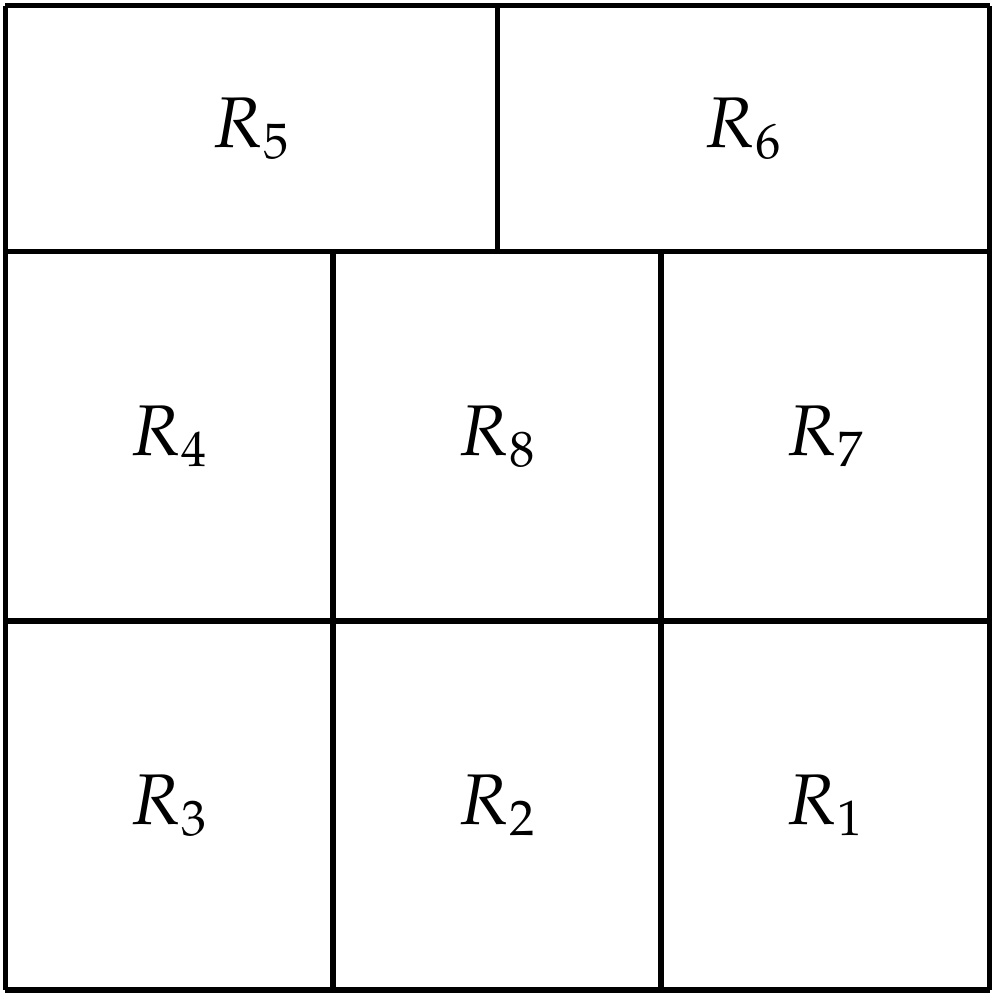} \\
	(a) & & (b)
	\end{tabular}
	\caption{Optimal decompositions of the unit square such that the maximum of the perimeters of the rectangles is minimized, for $k=7$ (a) and $k=8$ (b). Congruent rectangles are allowed in this case.}
	\label{F: Kong}
\end{figure}

We can compare these solutions with the optimal ones given by Kong, Mount and Werman. For $k=7$, the maximum of the perimeters is minimized at $22/14\approx 1.5714$ (corresponding to the perimeter of any of the congruent rectangles $R_1,R_2$ or $R_3$ in Figure \ref{F: Kong} (a)). For $k=8$, the solution is given by $17/12 \approx 1.4167$ (corresponding to the perimeter of any of the congruent rectangles $R_1,\dots,R_4,R_7,R_8$ in Figure \ref{F: Kong} (b)).

Although many decomposition problems have exact solutions, it seems that those involving non-congruent rectangles, like the ones featured in this section, are difficult to solve. Nevertheless, we hope to shed some light on them with the results presented here. 

\subsection*{Acknowledgments}
The research of C. Dalf\'o and N. L\'opez has been partially supported by grant MTM2017-86767-R (Spanish Ministerio de Ciencia e Innovaci\'on). The research of C. Dalf\'o and M. A. Fiol has been partially supported by
AGAUR from the Catalan Government under project 2017SGR1087 and by MICINN from the Spanish Government under project PGC2018-095471-B-I00. 


\begin{thebibliography}{10}
\bibitem{AK86}
N. Alon, and D. J. Kleitman,
Covering a square by small perimeter rectangles,
   {\em  Discrete Comput. Geom.}  {\bf 1} (1986) 1--7.

\bibitem{Bassen16}
H. Bassen, Further insight into the Mondrian art problem, 2016,
\url{https://mathpickle.com/mondrian-art-puzzles-solutions/} [Online accessed: 2019-10-01].
	
\bibitem{BBRR20}
O. Beaumont, V. Boudet, F. Rastello, and Y. Robert,
Partitioning a Square into Rectangles: NP-Completeness and Approximation Algorithms. [Research Report] LIP RR-2000-10, Laboratoire de l’informatique du parall\'elisme. 2000, 2+25 p. hal-02101984.
	
\bibitem{Brooks1987}
R. L. Brooks, C. A. B. Smith, A. H. Stone, and W. T. Tutte,
The dissection of rectangles into squares, pages 88--116,
Birkh{\"a}user Boston, Boston, MA, 1987.
	
\bibitem{DaFiLo2020}
C. Dalf\'o, M. A. Fiol, and N. L\'opez,
New results for the Mondrian art problem, 2020, submitted.
	
\bibitem{KMW87}
T. Y. Kong, D. M. Mount, and M. Werman,
The decomposition of a square into rectangles of minimal perimeter,
\textit{Discrete Appl. Math.} \textbf{16} (1987) 239--243.
	
\bibitem{okuhn2018mondrian}
C. O'Kuhn,
The Mondrian puzzle: A connection to number theory,
2018, {\tt arXiv:1810.04585}.
\end{thebibliography}

\end{document}